\documentclass[11pt]{article}

\usepackage[usenames,dvipsnames,svgnames,table]{xcolor}
\usepackage[titletoc,title]{appendix}
\usepackage{amsfonts}
\usepackage{graphicx}
\usepackage{hyperref}
\usepackage{amsmath}
\usepackage{amssymb}
\usepackage{framed}
\usepackage{amsthm}
\usepackage{comment}
\usepackage{cancel}
\usepackage{float}
\usepackage{listings}

\newtheorem{theorem}{Theorem}[section]
\newtheorem{proposition}[theorem]{Proposition}
\newtheorem{lemma}[theorem]{Lemma}
\newtheorem{claim}[theorem]{Claim}
\newtheorem{corollary}[theorem]{Corollary}

\newtheorem{definition}[theorem]{Definition}

\theoremstyle{definition}

\theoremstyle{definition}
\newtheorem*{remark*}{Remark}

\newtheorem*{proposition*}{Proposition}

\theoremstyle{definition}

\newcommand{\mrm}[1]{\mathrm{#1}}
\newcommand{\skipline}{$\phantom{}$}

\newcommand{\reals}{\mathbb{R}}
\newcommand{\spm}{\left\{-1,1\right\}}

\newcommand{\seq}{\subseteq}
\newcommand{\sm}{\setminus}

\newcommand{\defeq}{\mathrel{\mathop:}=}

\newcommand{\one}{\boldsymbol{1}}

\newcommand{\pr}{\Pr}
\newcommand{\given}[2]{#1\,\middle|\, #2}
\DeclareMathOperator*{\be}{\mathbb{E}}
\DeclareMathOperator*{\var}{\mrm{Var}}

\newcommand{\andd}{\wedge}

\newcommand{\wt}[1]{\widetilde{#1}}
\newcommand{\norm}[1]{\left\Vert {#1}\right\Vert}
\newcommand{\lnorm}[1]{\Vert {#1}\Vert}

\newcommand{\li}{\left}
\newcommand{\ri}{\right}

\newcommand{\cc}{\colon}

\newcommand{\set}[2]{\left\{ \given{#1}{#2}\right\}}
\newcommand{\prgiven}[2]{\pr\left[ \given{#1}{#2}\right]}

\newcommand{\linner}[2]{\langle {#1}, {#2} \rangle}
\newcommand{\beg}[3]{\begin{#1}\ifx\hxczuwx#2\hxczuwx\else\label{#2}\fi#3\end{#1}}
\newcommand{\bqed}{\hfill\blacksquare}

\usepackage{titlesec}

\titlespacing{\paragraph}{%
	0pt}{
	0.5\baselineskip}{
	1em}

\newcommand{\tops}[1]{\texorpdfstring{#1}{}}

\begin{document}
	
	\title{Slicing all Edges of an $n$-cube Requires $n^{2/3}$ Hyperplanes}

	\author{
		Ohad Klein\thanks{School of Computer Science and Engineering, Hebrew University of Jerusalem. \texttt{ohadkel@gmail.com}.\newline
		This research was supported in part by a grant from the Israel Science Foundation (ISF Grant No. 1774/20), and by a grant from the US-Israel Binational Science Foundation and the US National Science Foundation (BSF-NSF Grant No. 2020643).}
	}
	
	\maketitle
	
	\begin{abstract}
		Consider the $n$-cube graph with vertices $\spm^n$ and edges connecting vertices with hamming distance $1$. How many hyperplanes in $\reals^n$ are needed in order to dissect all edges? We show that at least $\wt{\Omega}(n^{2/3})$ are needed, which improves the previous bound of $\Omega(n^{0.51})$ by Yehuda and Yehudayoff.
	\end{abstract}
	

	\section{Introduction}
		Consider the hypercube graph $Q_n$ consisting of the vertices $\spm^{n} \subset \reals^n$ with edges connecting adjacent vertices $x,y$ satisfying $\sum_{i=1}^{n} |x_i-y_i| = 2$. It was asked how many hyperplanes in $\reals^n$ are required in order to dissect\footnote{A hyperplane dissects an edge if it intersects the edge, but does not contain any of its endpoints. In such a case we also say that the edge crosses the hyperplane. One may relax the notion to allow the hyperplane to include at most one of the endpoints -- our result will carry over and incur a multiplicative constant degradation.} all edges of $Q_n$ -- call this quantity $H(n)$. There are several natural configurations of hyperplanes involving $n$ hyperplanes, yielding $H(n)\leq n$. However, there is a configuration involving only $5$ planes in $6$ dimensions that was found by Paterson~\cite{Saks93}, implying that $H(n) \leq \lceil \frac{5n}{6}\rceil$. On the lower bound front, each hyperplane can dissect at most $\lceil n/2\rceil \cdot \binom{n}{\lceil n/2 \rceil}$ edges~\cite[Chapter~7.E]{Tit62}, out of $n \cdot 2^{n-1}$ -- hence at least $\Omega(\sqrt{n})$ hyperplanes are required. Even though the problem was posed in the 70s~\cite{ON71, Gr72}, only recently a better lower bound of $\Omega(n^{0.51})$ was given by Yehuda and Yehudayoff~\cite{YY21} (for the history of the problem and its applications, see references therein).
		In this paper we prove the following.
		\begin{theorem}\label{thm:main}
			\[
				H(n) \geq \Omega(n^{2/3} / \log(n)^{4/3}).
			\]
		\end{theorem}
		While it is likely that the inverse logarithmic factors can be improved, we are unaware of polynomial improvements. Some tools, especially the anti-concentration arguments are shared with~\cite{YY21}, but the general approach seems to be different.

	\subsection{Proof overview}\label{ssec:proof-overview}
		The lower bound $H(n) \geq \Omega(\sqrt{n})$ proceeds by showing each hyperplane dissects at most $O(1/\sqrt{n})$ of the available edges in $Q_n$.
		\paragraph{Example of middle layers.}
		In fact, once we surpass $O(\sqrt{n})$ hyperplanes, the planes can dissect almost all edges. To see this, partition all edges of the hypercube into $n+1$ layers, the $k$'th of which connects vertices of Hamming weight $k$ to those with weight $k+1$. Each such layer is dissected by a hyperplane orthogonal to the vector $(1,\ldots, 1)$. Most edges in the hypercube are contained in the $O(\sqrt{n})$ middle layers, thus a uniformly random edge is likely to cross some hyperplane in a fixed set of $O(\sqrt{n})$ planes.
		\newtheorem*{question*}{Question}
		\begin{question*}
		Is there a (natural) distribution over edges which is likely to \emph{evade} (i.e. not cross) a set of $m$ hyperplanes?
		\end{question*}
		In our example, the answer is positive -- choose a uniformly random layer in $Q_n$ (out of $n$), and draw a random edge $E$ within that layer. What about a general configuration of planes?
		\paragraph{Setting.}
		Assume $m$ hyperplanes are given, with the $\ell$'th hyperplane specified by
		\beg{equation}{eq:Hi}{
			H_\ell = \set{x \in \reals^n}{\linner{v_\ell}{x} = t_\ell},
		}
		with $v_{\ell} \in \reals ^n$ having $\norm{v_\ell}_2 = 1$ and $t_{\ell} \in \reals$.

		\paragraph{Middle layers revisited.}
		We present a different view for why in the example of the middle layers, the random edge $E$ is likely to evade a small set of hyperplanes having $v_\ell=\frac{1}{\sqrt{n}}(1,\ldots,1)$.
		When projecting the edge $E$ on $v_\ell$, we produce an interval of length $2/\sqrt{n}$ that is located anywhere on $(-\sqrt{n}, \sqrt{n}) \cdot v_\ell$ equally likely. Meanwhile each layer $H_\ell$ projected on $v_\ell$ is a single point. Hence $\Omega(n)$ hyperplanes with $v_\ell=\frac{1}{\sqrt{n}}(1,\ldots,1)$ are required in order to dissect all edges.

		\paragraph{The general case.} In general, projecting an edge on $v_\ell$ results in an interval of length $\leq 2\norm{v_\ell}_{\infty}$ parallel to $v_\ell$. Assuming $\norm{v_\ell}_{\infty} = \wt{O}(1/\sqrt{n})$ (which holds for a uniformly random vector $v_\ell$ with norm $1$), we have that each edge corresponds to an interval of length $\wt{O}(1/\sqrt{n})$ on $v_\ell$, while $H_\ell$ is projected to a point. Hence, if we are able to produce a random edge whose projection on any $v_\ell$ has a significant variance -- it will likely evade every $H_\ell$.

		\paragraph{The distribution.} Keeping the assumption $\norm{v_\ell}_{\infty} = \wt{O}(1/\sqrt{n})$ for the $m$ hyperplanes, let
		\[
			P = \sum_{\ell =1}^{m} \alpha_\ell\sqrt{\frac{n}{m}} v_\ell,
		\]
		with $\alpha_\ell \sim U([-1,1])$ independently. Note that each coordinate $P_i$ has $\be[P_i]= 0$ and $\var[P_i]=(n/3m) \sum_{\ell=1}^{m} v_{\ell i}^2= \wt{O}(1)$. Hence we (informally) imagine $P$ as a vertex in $\spm^n$. Moreover, the standard deviation of the (norm of the) projection of $P$ on $v_\ell$ is at least $\Omega(\sqrt{n/m})$.
		Considering a uniformly random edge incident to $P$, the probability that the projection of this edge on $v_\ell$ (which has length $\wt{O}(1/\sqrt{n})$) crosses the projection of $H_\ell$ on $v_\ell$ (which is a single point), is about $(1/\sqrt{n})/\sqrt{n/m} = \sqrt{m}/n$. Since there are $m$ hyperplanes, and we wish the edge to evade all of them, it is sufficient that $m \cdot \sqrt{m}/n \ll 1$, or in other words, $m \ll n^{2/3}$.

		\paragraph{Getting rid of the assumption.}
		Although abandoning the assumption $\norm{v_\ell}_{\infty} = \wt{O}(1/\sqrt{n})$ introduces a few technicalities in the proof, it does not severely affect the general approach.
		We get $v_\ell$ decomposed into different groups of coordinates, such that $v_\ell$ is roughly constant on each group. We call this process a `binary decomposition'. Then, we slightly alter the definition of $P$ and the rest of the proof generalizes smoothly.

		\paragraph{Organization.} In Sections~\ref{ssec:decomp}-\ref{ssec:evasive} we define a distribution over edges. Sections~\ref{ssec:linear}-\ref{ssec:anti} introduce linear forms and some of their anti-concentration properties. Finally, Section~\ref{ssec:glue} shows the distribution evades the hyperplanes. Appendix~\ref{app:simple-proofs} contains proofs of standard facts that are used in the paper.

	\section{Proof}
		\subsection{A random bounded vector correlating with all \tops{$v_\ell$}}\label{ssec:decomp}
		The proof overview in Section~\ref{ssec:proof-overview} assumes that $\norm{v_\ell}_{\infty}$ is small. The following definition is useful in order to deal with general vectors $v_\ell$.
		\begin{definition}[Vector Binary Decomposition]\label{def:decompose}
			Let $v \in \reals^n$. For an integer $j$, let $K(j)$ be the set of coordinates of $v$ whose absolute value is in the range $(2^{-j-1} , 2^{-j}]$, that is
			\[
				K(j) = \set{k\in [n]}{ 2^{-j-1} < |v_k| \leq 2^{-j} }.
			\]
			Let $J$ be these $j$'s for which $K(j)$ is nonempty. The Binary decomposition of $v$ is 
			\[
				v = \sum_{j \in J} v^{(j)},
			\]
			where $v^{(j)} \in \reals^n$ aggregates all coordinates $K(j)$, that is
			\[
				v^{(j)}_k = v_k \cdot \one \li\{ k \in K(j) \ri\}.
			\]
		\end{definition}

		\begin{definition}[The Random Bias]\label{def:bias-rv}
			Let $v_1, \ldots, v_m$ be vectors in $\reals^n$ having $\norm{v_\ell}_2=1$ for all $\ell$.
			Let $v_\ell=\sum_{j \in J_\ell} v_\ell^{(j)}$ be the Binary decomposition of $v_\ell$, according to Definition~\ref{def:decompose}.
			Define the random variable
			\[
				P \defeq \frac{1}{10\sqrt{m\log(n)}}\sum_{\ell =1}^{m} \sum_{j \in J_\ell} \alpha_{\ell j} 2^{j} v_\ell^{(j)}
			\]
			with $\alpha_{\ell j} \sim U([-1, 1])$ uniformly and independently distributed in the interval $[-1, 1]$.
		\end{definition}
		
		We will use $P$ in order to produce a product distribution over $\spm^n$ which has mean $P$. For the sake of this distribution not to be too biased, we note that with high probability $\lnorm{P}_{\infty} \leq 1/2$; the simple proof appears in Appendix~\ref{app:simple-proofs}.
		\begin{lemma}\label{lem:Linfty}
			Let $P$ be distributed as in Definition~\ref{def:bias-rv}, then
			\[
				\pr[\lnorm{P}_\infty > 1/2] \leq 2/n,
			\]
			where the randomness is taken over the distribution of $\alpha_{\ell j}$'s.
		\end{lemma}

		\subsection{Hyperplane-evasive edges}\label{ssec:evasive}
		
		\begin{definition}
			Given a vector $p \in [-1, 1]^n$, we define a probability distribution $\mu_p$ on vectors $z \in \spm^n$, by letting $z_i$ be independent random variables with
			\[
				\pr[z_i = 1] = \frac{1+p_i}{2}, \qquad \pr[z_i = -1] = \frac{1-p_i}{2}.
			\]
		\end{definition}

		We now define a distribution on edges of $Q_n$ which are likely to evade every $H_\ell$.
		\begin{definition}\label{def:evasive}
			Let $H_i, v_i, t_i$ be as in Section~\ref{ssec:proof-overview},
			and let $P$ be a random variable which is distributed according to Definition~\ref{def:bias-rv}, conditioned on that $\lnorm{P}_{\infty} \leq 1/2$.

			Let $U$ be a random vertex of $Q_n$ drawn from $\mu_P$, and let $k \sim [n]$ be uniformly distributed. We define $(U,k)$ as the edge which is incident to $U$ and is parallel to the $k$'th axis.
		\end{definition}

		Our main technical result is the following proposition, which states that the probability of the edge $(U, k)$ to cross any of the hyperplanes $H_i$ is small.
		\begin{proposition}\label{prop:bound-halfspace}
			Let $(U,k)$ be distributed as in Definition~\ref{def:evasive}. Then for all $\ell$,
			\beg{equation}{eq:bound-halfspace}{
				\Pr[(U,k)\text{ crosses } H_\ell] \leq O\li(  \sqrt{m}\log(n)^2/n  \ri).
			}
		\end{proposition}
		\begin{corollary}[Restatement of Theorem~\ref{thm:main}]
			\[
				H(n) \geq \Omega(n^{2/3} / \log(n)^{4/3}).
			\]
		\end{corollary}
		\begin{proof}
			Suppose the hyperplanes $H_1, \ldots, H_m$ from~\eqref{eq:Hi} dissect all edges of $Q_n$. Since $(U,k)$ is a distribution over edges of $Q_n$, we have $\pr[\exists i \cc (U,k)\text{ crosses } H_i] = 1$.
			However, by union bound and Proposition~\ref{prop:bound-halfspace},
			\[
				\pr[\exists i \cc (U,k)\text{ crosses } H_i] \leq \sum_{i=1}^{m} \pr[(U,k)\text{ crosses } H_i] \leq
				O\li(
					m^{3/2} \log(n)^{2} / n
				\ri).
			\]
			This yields $m \geq \Omega(n^{2/3} / \log(n)^{4/3})$.
		\end{proof}

		We prove Proposition~\ref{prop:bound-halfspace} using anti-concentration techniques. In order to phrase our result in a probabilistic setting, we use the following claim.
		\begin{claim}\label{clm:inter-cond}
			In order for an edge $(u, k)$ to cross the hyperplane $H = \set{x \in \reals^n}{\linner{v}{x} = t}$ we must have
			\begin{equation}\label{eq:inter-cond}
				|\linner{v}{u} - t| < 2|v_k|.
			\end{equation}
		\end{claim}
		\begin{proof}
			Let $u'$ be the vertex connected to $u$ by $(u,k)$. That is, $u'$ is either $u+2e_k$ or $u-2e_k$. In order for the edge $(u, k)$ to cross $\set{x \in \reals^n}{\linner{v}{x} = t}$ we must have that $\linner{v}{u}-t$ and $\linner{v}{u'}-t$ have different signs (and are nonzero). Since $\linner{v}{u'} = \linner{v}{u\pm 2e_k} = \linner{v}{u}\pm 2v_k$, the edge $(u,k)$ may cross $H$ only if~\eqref{eq:inter-cond} holds.
		\end{proof}

		\subsection{Linear forms}\label{ssec:linear}
		\begin{definition}[Biased Linear Form]\label{def:blf}
			Let $v, p \in \reals^n$ have $\lnorm{p}_{\infty} \leq 1$ and let $x \sim \mu_p$. We define the associated biased linear form to be the random variable
			\[
				X_{v, p} \defeq \linner{v}{x}.
			\]
		\end{definition}
		We recall a standard tail inequality for linear forms.
		\begin{claim}[Chernoff-Hoeffding inequality~\cite{Hoeff63}]\label{clm:chernoff}
			Let $p,v \in \reals^n$ have $\lnorm{p}_{\infty} \leq 1$. Then for any $\sigma > 0$,
			\beg{equation}{eq:chernoff}{
				\Pr_{x \sim \mu_p}[ |X_{v,p} - \linner{v}{p}| > \sigma \lnorm{v}_2 ] \leq 2\exp(-\sigma^2/2).
			}
		\end{claim}

		\subsection{Anti-concentration of linear forms}\label{ssec:anti}

		\begin{definition}[L\'{e}vy Concentration function]\label{def:levy}
			Let $X$ be a random variable. It's $\alpha$-concentration is defined as
			\[
				Q(\alpha, X) \defeq \sup_{t \in \reals} \Pr[ |X - t| < \alpha].
			\]
		\end{definition}
		We need a basic property of concentration functions.
		\begin{claim}\label{clm:levy-simple}\skipline
			Every random variable $X$ and integer $k \geq 1$ satisfies $Q(k\alpha, X)\leq k Q(\alpha, X)$.
		\end{claim}
		\begin{proof}
			Follows by (almost) covering an interval of length $2k\alpha$ using $k$ intervals of length $2\alpha$ (and taking a limit).
		\end{proof}
		Next, we need a simple Littlewood-Offord-type result regarding the anti-concentration of a biased linear form $X$. For example, it follows from~\cite{Rog61}. For completeness, we include a quick proof in Appendix~\ref{app:simple-proofs}.
		\begin{lemma}\label{lem:littlewood}
			There is a universal constant $C \geq 1$, such that the following holds. Let $X_{v,p}$ be a biased linear form with $\lnorm{p}_{\infty} \leq 1/2$.
			Suppose that $\alpha$ satisfies $\alpha \leq \min\{|v_1|, \ldots, |v_a|\}$, then
			\[
				Q(\alpha, X_{v, p}) \leq C /\sqrt{a}.
			\]
		\end{lemma}
		While Lemma~\ref{lem:littlewood} gives at most an $O(1/\sqrt{n})$ bound on $Q(\alpha, X_{v,p})$ we can sometimes exponentially improve this bound, as the following lemma demonstrates.
		\beg{lemma}{lem:group}{
			Let $v, p\in \reals^n$ have $\lnorm{p}_{\infty} \leq 1/2$.
			Suppose $v$ has a binary decomposition
			\[
				v = \sum_{j \in J} v^{(j)},
			\]
			with the corresponding partition of the nonzero coordinates as $\bigcup_{j \in J} K(j)$.
			Let $\alpha > 0$ have at least $2r\log(n)$ elements $j \in J$ with $2^{-j-1} \geq \alpha$, where $r \geq 0$ is an integer.
			\begin{enumerate}
				\item Let $C$ be the universal constant from Lemma~\ref{lem:littlewood}. If $n \geq 10\exp(10C^2)$, then
				\beg{equation}{eq:negl}{
					Q(\alpha, X_{v, p}) \leq 2^{-r}.
				}
				\item For all $n$,
				\beg{equation}{eq:negl2}{
					Q(\alpha, X_{v, p}) \leq O(2^{-r}).
				}
			\end{enumerate}
		}
		\begin{proof}
		The proof of~\eqref{eq:negl} is by induction on $r$, where the $r=0$ case is trivial.

		Let $L=\lceil 4C^2 \rceil$ and $j_1 < \ldots < j_{L}$ be the $L$ largest $j \in J$ satisfying $2^{-j-1} \geq \alpha$. Denote
		\[
			\alpha' \defeq \sum_{j \geq j_1} \lnorm{v^{(j)}}_1.
		\]
		\begin{claim}\label{clm:alpha}
		 	There are at least $2(r-1)\log(n)$ elements $j \in J$ with $2^{-j-1} \geq 2\alpha'$.
		\end{claim}
		To prove Claim~\ref{clm:alpha}, observe that $2\alpha' \leq 2n \cdot 2^{-j_1}$, hence there are at most $\log_2(2n)$ elements $j\in J$ with $j < j_1$ and $2^{-j-1} < \alpha'$. Including $j_1, \ldots, j_{L}$, there are at most $\log_2(2n)+L$ elements $j \in J$ with $2^{-j-1} \in [\alpha, \alpha')$. Since $n \geq 10\exp(10C^2)$, we have $\log_2(2n)+L \leq 2\log(n)$, thus establishing Claim~\ref{clm:alpha}.$\bqed$

		Given a decomposition $v = u+w$ we have $X_{v,p} = X_{u,p} + X_{w,p}$. We set
		\[
			u \defeq \sum_{j\in J \cc j < j_1} v^{(j)}, \qquad w \defeq \sum_{j\in J \cc j \geq j_1} v^{(j)}.
		\]
		Since $u$ and $w$ are supported on different coordinates, we have that $X_{u,p}$ and $X_{w,p}$ are independent random variables.

		We use the induction hypothesis~\eqref{eq:negl} with $2\alpha'$ instead of $\alpha$ and obtain
		\[
			Q(2\alpha', X_{u, p}) \leq 2^{1-r}.
		\]
		We immediately complete the inductive proof of~\eqref{eq:negl} using the following claim.
		\begin{claim}\label{clm:induction}
			$Q(\alpha, X_{v,p}) \leq \frac{1}{2} Q(2\alpha', X_{u, p})$.
		\end{claim}
		In order to prove Claim~\ref{clm:induction} we show the following for all $t \in \reals$
		\[
			\pr[ |X_{v,p} - t| < \alpha] \leq \frac{1}{2}\pr[ |X_{u,p} - t| < 2\alpha'].
		\]
		Indeed, note that by definition of $\alpha'$ we always have $|X_{w,p}| \leq \alpha'$. Hence, if $|X_{v,p} - t| < \alpha$ then necessarily $|X_{u,p} - t| < \alpha' + \alpha < 2\alpha'$. Conditioned on any specific value of $X_{u,p}$, we also need that
		\beg{equation}{eq:Qw}{
			|X_{w,p} - (t - X_{u,p})| = |X_{v,p} - t| < \alpha.
		}
		However, since $X_{w,p}$ is independent of $X_{u,p}$, the probability that~\eqref{eq:Qw} holds is upper bounded by $Q(\alpha, X_{w,p})$. Therefore, $Q(\alpha, X_{v,p}) \leq Q(2\alpha', X_{u, p}) \cdot Q(\alpha, X_{w, p})$, and Claim~\ref{clm:induction} follows from
		\beg{equation}{eq:Qw-aux}{
			Q(\alpha, X_{w, p}) \leq 1/2.
		}
		To confirm~\eqref{eq:Qw-aux} note that $w$ contains at least $L$ coordinates greater than $\alpha$, which reside in the non-empty sets $K(j_1), \ldots, K(j_L)$, hence Lemma~\ref{lem:littlewood} implies $Q(\alpha, X_{w, p}) \leq C/\sqrt{L} \leq 1/2$.
		$\bqed$

		The proof of~\eqref{eq:negl} is complete. Finally, observe that~\eqref{eq:negl2} follows from~\eqref{eq:negl} -- by plugging in a large enough constant in the right hand side of~\eqref{eq:negl2}, it holds even when $n < 10\exp(10C^2)$.
		\end{proof}

		\subsection{Evasiveness of Definition~\ref{def:evasive}}\label{ssec:glue}
		\beg{lemma}{lem:glue}{
			Let $\bar{v}_1, \ldots, \bar{v}_m$ be vectors in $\reals^n$ with $\norm{\bar{v}_i}_2 = 1$, and suppose that $P$ is distributed according to Definition~\ref{def:bias-rv}, conditioned on $\lnorm{P}_{\infty} \leq 1/2$.
			If $x \sim \mu_P$, then for $v=\bar{v}_1$ and $t \in \reals$ we have
			\beg{equation}{eq:glue}{
				\sum_{k = 1}^{n} \pr_{x \sim \mu_P}[|\linner{v}{x} - t| < 2 |v_{k}|] \leq O(\sqrt{m} \log(n)^2).
			}
		}
		This lemma immediately implies our main technical result.
		\begin{proof}[Proof of Proposition~\ref{prop:bound-halfspace}]
			Since $k$ in Definition~\ref{def:evasive} is uniformly random (out of $n$ options), Lemma~\ref{lem:glue} implies~\eqref{eq:bound-halfspace} through Claim~\ref{clm:inter-cond}.
		\end{proof}

		\begin{proof}[Proof of Lemma~\ref{lem:glue}]
			Suppose $v=\bar{v}_1$ has the binary decomposition
			\[
				v = \sum_{j \in J} v^{(j)},
			\]
			with the corresponding partition of the nonzero coordinates as $\bigcup_{j \in J} K(j)$.
			We show the following. If $i \in J$ has at least $2r\log(n)$ elements in $J$ smaller than $i$, with $r \geq 0$ an integer, then
			\beg{equation}{eq:group-total}{
				\sum_{k \in K({i})} \pr_{x \sim \mu_P}[|\linner{v}{x} - t| < 2 |v_{k}|] \leq
				O\li(
					2^{-r} \sqrt{m} \log(n) + \frac{|K({i})|}{n}
				\ri).
			}
			Note that the probability is taken both with respect to $P$ and $x\sim \mu_P$.
			Let us see that~\eqref{eq:group-total} implies~\eqref{eq:glue}:
			\[
			\begin{aligned}
				\sum_{k = 1}^{n} \pr_{x \sim \mu_P}[|\linner{v}{x} - t| < 2 |v_{k}|]
				& = 
				\sum_{j \in J} \sum_{k \in K({j})} \pr_{x \sim \mu_P}[|\linner{v}{x} - t| < 2 |v_{k}|]
				\\ & \leq
				O\li(
					\sum_{j \in J} 2^{-r(j)} \sqrt{m} \log(n) + \frac{|K({j})|}{n}
				\ri)
				\\ & \leq
				O\li( \sqrt{m} \log(n)^2  \ri) + O(1),
			\end{aligned}
			\]
			with $r(j)$ the maximal $r$ corresponding to $j$. Note that $\sum_{j\in J} 2^{-r(j)} = O(\log(n))$ as a sum of a geometric progression with each term appearing $\approx 2\log(n)$ times. The estimate~\eqref{eq:glue} follows.

			\medskip\noindent \textbf{Proving ~\eqref{eq:group-total}.}
			We let $\sigma \defeq 2+2\sqrt{\log(n)}$ and define the events
			\[
				E \defeq E_1 \cup E_2 \defeq \li\{ |\linner{v - v^{(i)}}{x} - t| \geq 3\lnorm{v^{(i)}}_1 \ri\} \cup \li\{ |\linner{v - v^{(i)}}{x} + \linner{v^{(i)}}{P} - t| \geq \sigma\lnorm{v^{(i)}}_2 \ri\},
			\]
			that split the probability space into $E$ and $\neg{E}$. We prove later that
			\beg{equation}{eq:negl-Inf}{
				\forall k \in K(i) \cc
				\prgiven{|\linner{v}{x} - t| < 2 |v_{k}|}{E}
				\leq O\li(
					1/n
				\ri).
			}
			and
			\beg{equation}{eq:negl-prob}{
				\pr[\neg{E}] \leq O\li(
					2^{-r}\sigma \sqrt{\frac{m\log(n)}{|K({i})|}}
				\ri),
			}
			This allows us to bound the terms on the left hand side of~\eqref{eq:group-total} as
			\beg{equation}{eq:total-prob}{
			\begin{aligned}
				\pr\li[ |\linner{v}{x} - t| < 2 |v_{k}| \ri]
				& =
				\prgiven{|\linner{v}{x} - t| < 2 |v_{k}|}{E} \pr[E] + 
				\prgiven{|\linner{v}{x} - t| < 2 |v_{k}|}{\neg{E}} \pr[\neg{E}]
				\\ & \leq
				O(1/n) +
				\prgiven{|\linner{v}{x} - t| < 2 |v_{k}|}{\neg{E}}
				\cdot
				O\li(
					2^{-r}\sigma \sqrt{\frac{m\log(n)}{|K({i})|}}
				\ri),
			\end{aligned}
			}
			for all $k \in K(i)$.
			Observe that conditioning on the value of $P$, the event $E$ is independent of the $K(i)$ coordinates of $x$. Hence, under any fixing of $P=p$,
			\beg{equation}{eq:Eindep}{
				\pr_{x\sim \mu_p}\li[ \given{|\linner{v}{x} - t| < 2 |v_{k}|}{\neg{E}, P=p} \ri] \leq Q(2 |v_{k}|, X_{v^{(i)}, p}) \leq 2Q(|v_{k}|, X_{v^{(i)}, p}),
			}
			using Claim~\ref{clm:levy-simple}.
			Assuming $|v_k|$ is the $a$'th largest coordinate in $v^{(i)}$, Lemma~\ref{lem:littlewood} implies that
			$Q(|v_k|, X_{v^{(i)}, P}) \leq O(1/\sqrt{a})$ (recall that we conditioned on $\lnorm{P}_{\infty} \leq 1/2$). Together with~\eqref{eq:Eindep} this yields
			\[
				\prgiven{|\linner{v}{x} - t| < 2 |v_{k}|}{\neg{E}} = \be_{P}\li[ \pr_{x\sim \mu_P}\li[ \given{|\linner{v}{x} - t| < 2 |v_{k}|}{\neg{E}, P} \ri] \ri] \leq O(1/\sqrt{a}).
			\]
			Computing the sum over all $k \in K(i)$ we refine~\eqref{eq:total-prob} into
			\[
			\begin{aligned}
				\sum_{k \in K(i)} \pr\li[ |\linner{v}{x} - t| < 2 |v_{k}| \ri]
				&\leq
				\sum_{a=1}^{|K(i)|} \li( O(1/n) + O(1/\sqrt{a}) \cdot O\li(
					2^{-r}\sigma \sqrt{\frac{m\log(n)}{|K({i})|}}
				\ri) \ri)
				\\ & =
				O\li(
				\frac{|K(i)|}{n} + 2^{-r} \sqrt{m} \log(n)
				\ri),
			\end{aligned}
			\]
			yielding~\eqref{eq:group-total}. It remains to verify~\eqref{eq:negl-Inf} and~\eqref{eq:negl-prob}.

			\medskip\noindent \textbf{Proving ~\eqref{eq:negl-Inf}.}
			Recall $E$ is the union of the two events $E_1$ and $E_2$. Under $E_1$ we actually have the stronger
			\beg{equation}{eq:negl_E1}{
				\prgiven{|\linner{v}{x} - t| < 2 |v_{k}|}{E_1} = 0.
			}
			Indeed, if $E_1$ happens then
			\[
			\begin{aligned}
				\li| \linner{v}{x} - t \ri| &=
				\li|  \linner{v - v^{(i)}}{x} + \linner{v^{(i)}}{x} - t \ri|
				\\
				& \geq 
				\li|  \linner{v - v^{(i)}}{x} - t \ri| - \li| \linner{v^{(i)}}{x} \ri|
				\\
				& \geq 
				3\lnorm{v^{(i)}}_1 - \lnorm{v^{(i)}}_1 \\
				& = 2\lnorm{v^{(i)}}_1 \geq 2|v_{1k}|,
			\end{aligned}
			\]
			giving~\eqref{eq:negl_E1}. We now prove~\eqref{eq:negl-Inf} with $E_2$ instead of $E$. If $E_2$ happens, then $|\linner{v}{x} - t| < 2 |v_{k}|$ means
			\[
			\begin{aligned}
				\li| \linner{v^{(i)}}{P} - \linner{v^{(i)}}{x} \ri|
				& \geq \li| \linner{v^{(i)}}{P} + \linner{v - v^{(i)}}{x} - t \ri| - |\linner{v}{x} - t|
				\\ & > \sigma\lnorm{v^{(i)}}_2 - 2|v_k|
				\geq (\sigma-2)\lnorm{v^{(i)}}_2,
			\end{aligned}
			\]
			by the triangle inequality. 
			Using again the fact that under conditioning of the value of $P$, the $K(i)$ coordinates of $x$ are independent of $E_2$, together with Hoeffding's inequality, we deduce
			\[
			\begin{aligned}
				\prgiven{|\linner{v}{x} - t| < 2 |v_{k}|}{E_2}
				& = \be_{P}\li[ \prgiven{|\linner{v}{x} - t| < 2 |v_{k}|}{E_2, P} \ri]
				\\ &=
				\be_{P} \li[ \pr \li[ \given{\li| \linner{v^{(i)}}{P} - \linner{v^{(i)}}{x} \ri| > (\sigma-2) \lnorm{v^{(i)}}_2}{P} \ri] \ri]
				\\ &
				\leq 2 \exp( -(\sigma-2)^2 / 2) = 2 / n^2.
			\end{aligned}
			\]
			This concludes the proof of~\eqref{eq:negl-Inf}.

			\medskip\noindent \textbf{Proving~\eqref{eq:negl-prob}.} Unlike the rest of the proof, we henceforth assume that $P$ has the exact same distribution as in Definition~\eqref{def:bias-rv}, and we are explicit about the conditioning on $\lnorm{P}_{\infty} \leq 1/2$.

			Let $F'$ be the event that $\{\lnorm{P}_{\infty} \leq 1/2\}$ and let $F$ be the event that $\{\lnorm{P}_{\infty, [n] \sm K(i)} \leq 1/2\}$, meaning that all coordinates of $P$ except those in $K(i)$ are less than $1/2$ in absolute value.

			In~\eqref{eq:negl-prob} we seek to bound $\prgiven{\neg{E}}{F'}$. However, it is more convenient to bound $\prgiven{\neg{E}}{F}$ instead. There is not much loss in doing so, in virtue of Lemma~\ref{lem:Linfty} and that $F' \seq F$:
			\[
				\prgiven{\neg{E}}{F'} = \frac{\pr\li[\neg{E} \andd F' \ri]}{\pr[F']} \leq \frac{\pr\li[\neg{E} \andd F \ri]}{(1-2/n) \pr[F]} \leq O(\prgiven{\neg{E}}{F}).
			\]
			Thus~\eqref{eq:negl-prob} is reduced to
			\beg{equation}{eq:negp}{
				\prgiven{\neg{E}}{F}
				\leq
				O\li(
					2^{-r}\sigma \sqrt{\frac{m\log(n)}{|K({i})|}}
				\ri)
			}
			The estimate~\eqref{eq:negp} follows by that $E=E_1\cup E_2$ hence $\prgiven{\neg{E}}{F}= \prgiven{\neg{E_1}}{F} \prgiven{\neg{E_2}}{F,\neg{E_1}}$, and
			\beg{equation}{eq:E_1}{
				\prgiven{\neg{E_1}}{F} \leq O(2^{-r}),
			}
			\vspace{-0.5cm}
			\beg{equation}{eq:E_2}{
				\prgiven{\neg{E_2}}{F,\neg{E_1}} \leq O\li(
					\sigma \sqrt{\frac{m\log(n)}{|K({i})|}}
				\ri).
			}
			To confirm~\eqref{eq:E_1} we use Lemma~\ref{lem:group}. Denote $\alpha\defeq 3\lnorm{v^{(i)}}_1$ and note $\alpha \leq 3n \cdot 2^{-i}$. Recalling the definition of $r$, we see that there are at least $2(r-1)\log(n)$ elements $j \in J$ with $2^{-j-1} \geq \alpha$. Hence~\eqref{eq:negl2} reads as $Q(\alpha, X_{u,P}) \leq O(2^{1-r})$ with $u=v-v^{(i)}$, thus implying~\eqref{eq:E_1}.

			In order to prove~\eqref{eq:E_2} we recall that by definition~\ref{def:bias-rv}
			\[
			\begin{gathered}
				P = \frac{1}{10\sqrt{\log(n)}} \sum_{\ell =1}^{m} \sum_{j \in J_\ell} \alpha_{\ell j} \frac{2^j}{\sqrt{m}} \bar{v}_\ell^{(j)},
				\\
				(\neg{E_2}) = \li\{|\linner{v - v^{(i)}}{x} + \linner{v^{(i)}}{P} - t| < \sigma\lnorm{v^{(i)}}_2 \ri\},
			\end{gathered}
			\]
			with $\alpha_{\ell j} \sim U([-1,1])$. Roughly speaking, we show that the random variable $\linner{v^{(i)}}{P}$ has a significant variance, and hence $\neg{E_2}$ is a rare event.

			Fix (that is, condition on) any value of $\alpha_{\ell j}$ whenever $(\ell, j) \neq (1, i)$ (recall $v=\bar{v}_1$) and on all coordinates of $x$ that are outside $K(i)$. Observe that the occurrence of the events $E_1$ and $F$ is determined by this fixing. However, since $\alpha_{1 i}$ is independent of all other $\alpha_{\ell j}$, and it affects only the $K(i)$ coordinates of $P$ (hence is independent also of $F$ and $E_1$), $\alpha_{1 i}$ is still uniformly distributed in $[-1, 1]$. Hence, the following is a variable which is uniformly distributed in some real interval
			\begin{equation}\label{eq:uniform-var}
				\linner{v^{(i)}}{P} = \alpha_{1 i} \frac{2^{i}}{10\sqrt{m\log(n)}} \linner{v^{(i)}}{v^{(i)}} + \rho,
			\end{equation}
			with $\rho \in \reals$ a constant depending on the fixing. Note that by definition of binary decomposition, $v^{(i)}$ has $K(i)$ nonzero coordinates, all having absolute values in the range $(2^{-i-1}, 2^{-i}]$. Hence $\lnorm{v^{(i)}}_2^2 = \Theta(|K(i)| 4^{-i})$. Overall, the probability that $\neg{E_2}$ happens -- which means that $\linner{v^{(i)}}{P}$ is only $\sigma\lnorm{v^{(i)}}_2$ away from the constant $(t-\linner{v - v^{(i)}}{x})$ -- is bounded by the anti-concentration of the uniform variable~\eqref{eq:uniform-var}:
			\[
				\prgiven{\neg{E_2}}{F,\neg{E_1}} \leq \frac{\sigma\lnorm{v^{(i)}}_2}{\frac{2^{i}}{10\sqrt{m\log(n)}}\linner{v^{(i)}}{v^{(i)}}} = \frac{10\sigma \sqrt{m\log(n)}}{2^{i}\lnorm{v^{(i)}}_2} = \frac{10\sigma \sqrt{m}{\log(n)}}{\sqrt{|K(i)|}},
			\]
			yielding~\eqref{eq:E_2} and concluding the proof.
		\end{proof}

		\section*{Acknowledgements}
			The author is grateful to Nathan Keller for matchmaking him with the problem. I thank Yiting Wang for spotting an inaccuracy in an earlier version of this paper. The help of Yotam Shomroni regarding the presentation is highly appreciated.


	\begin{appendices}
	\section{Proofs of standard facts}\label{app:simple-proofs}
	\begin{proof}[Proof of Lemma~\ref{lem:Linfty}]
			Consider $P_i$, it is a random variable equal to
			\[
				\frac{1}{10\sqrt{m\log(n)}} \sum_{\ell =1}^{m} \alpha_\ell v_{\ell i}'
			\]
			with $v_{\ell i}' \in (1/2, 1]$ being $v_{\ell i}$ multiplied by an integral power of 2, and $\alpha_\ell \sim U([-1, 1])$.
			Note that each $\alpha_\ell$ may be implemented as $\sum_{j=1}^{\infty} 2^{-j} x_j$ with $x_j \sim \spm$ uniformly and independently.
			If we hence aggregate an infinite vector $u=(2^{-j} v_{\ell i}')_{\ell =1, j=1}^{m, \infty}$, then $P_i$ has the distribution of $\frac{1}{10\sqrt{m\log(n)}} \linner{u}{x}$ with all $x_i\sim \spm$ independently and uniformly distributed. Note also that $\lnorm{u}_2 = \sum_{\ell =1}^{m} v_{\ell i}'^2 / 3 \leq m/3$. Therefore, Claim~\ref{clm:chernoff} implies
			\[
				\pr_{\alpha}[|P_i| \geq 1/2] = \pr_x[|\linner{u}{x}| \geq 5\sqrt{m\log(n)}] \leq \pr_x[|\linner{u}{x}| \geq 8\lnorm{u}_2\sqrt{\log(n)}] \leq 2/n^{32}.
			\]
			Union bound implies that $\lnorm{P}_{\infty} \leq 1/2$ except for probability $2/n^{31}$.
	\end{proof}

	\begin{proof}[Proof of Lemma~\ref{lem:littlewood}]\skipline

			\noindent\textbf{Case $p=\underline{0}$.} In this case $X_{v,p} = \sum_{i=1}^{n} x_i v_i$ with the $x_i$'s uniformly and independently distributed in $\spm$. Without loss of generality we may assume $v_i \geq 0$ for all $i$. For any $t \in \reals$ the set of $x$'s for which $|X_{v,p} - t| < \alpha$ corresponds to an anti-chain in every subcube of $\spm^n$, generated by the $[a]$ coordinates. Sperner's theorem~\cite{Sperner28} hence implies $\pr[|X_{v,p} - t| < \alpha] \leq \binom{a}{\lfloor a/2 \rceil}/2^a = O(1/\sqrt{a})$, yielding $Q(\alpha, X_{v, p}) \leq O(1/\sqrt{a})$.

			\noindent\textbf{Case $\lnorm{p}_\infty \leq 1/2$.}
			We present $X_{v,p}$ as a mixture of random variables $X_{v', p'}$ with $p'=\underline{0}$, and with $v'$ usually having $\Omega(a)$ coordinates $\geq a$, and the previous case of the proof.

			Indeed, choose $v'$ to be $v$ with each coordinate $i$ zeroed out independently with probability $|p_i|$. If $v_i'$ was zeroed out, choose $\rho_i$ to be $v_i \cdot \frac{p_i}{|p_i|}$, otherwise (or if $p_i=0$) -- $\rho_i = 0$. We note that if $x\sim \mu_p$ and $x' \sim \mu_{\underline{0}}$, then $v_i \cdot x_i$ can be realized as $\rho_i + v_i'\cdot x'_i$ -- the excess probability associated with one of the outcomes of $x_i \in \spm$ is simulated through $\rho_i$.

			We conclude that the distribution of $X_{v,p}$ can be sampled by first drawing $v', \rho$ and then sampling $\sum_i v_i' \cdot x'_i + \sum_i \rho_i$. In particular, 
			\beg{equation}{eq:littleuse}{
				Q(\alpha, X_{v,p}) \leq \be[Q(\alpha, X_{v',\underline{0}})] \leq \be[\min(O(1/\sqrt{a'}), 1)],
			}
			with $a'$ the number of coordinates $k$ with $|v'_k| \geq \alpha$. By definition of $v'$, $a'$ is distributed as a sum of $a$ Bernoulli random variables, the $i$'th of which having success probability $1-|p_i|$. Since $|p_i| \leq 1/2$, we have $\be[a'] \geq a/2$. Moreover, the distribution of $a'$ is quite concentrated around $\be[a']$, with Chebyshev's inequality implying that $\pr[a' < a/4] \leq 4/a$. This with~\eqref{eq:littleuse} give
			\[
				Q(\alpha, X_{v,p}) \leq O(\be[\min(1/\sqrt{a'}, 1)]) \leq O(1/\sqrt{a/4} + 4/a) = O(1/\sqrt{a}).
			\]
	\end{proof}
	\end{appendices}
\end{document}